\documentclass[10pt]{article}

\usepackage{latexsym,amsfonts,amsmath,amssymb,mathrsfs,url,color,amsthm}
\usepackage{graphicx}

\newtheorem{thm}{Theorem}
\newtheorem{lem}{Lemma}
\newtheorem{rem}{Remark}
\newtheorem{prop}{Proposition}
\newtheorem{defn}{Definition}

\newcommand{\rd}{\,\mathrm{d}}
\newcommand{\bsc}{\boldsymbol{c}}
\newcommand{\bsk}{\boldsymbol{k}}
\newcommand{\bsx}{\boldsymbol{x}}
\newcommand{\bsy}{\boldsymbol{y}}
\newcommand{\bszero}{\boldsymbol{0}}

\newcommand{\nat}{\mathbb{N}}
\newcommand{\RR}{\mathbb{R}}

\newcommand{\ZZ}{\mathbb{Z}}

\allowdisplaybreaks

\begin{document}

\title{On the $L_p$ discrepancy of two-dimensional folded Hammersley point sets}

\author{Takashi Goda\thanks{Graduate School of Engineering, The University of Tokyo, 7-3-1 Hongo, Bunkyo-ku, Tokyo 113-8656, Japan (\tt{goda@frcer.t.u-tokyo.ac.jp})}}

\date{\today}

\maketitle

\begin{abstract}
We give an explicit construction of two-dimensional point sets whose $L_p$ discrepancy is of best possible order for all $1\le p\le \infty$. It is provided by folding Hammersley point sets in base $b$ by means of the $b$-adic baker's transformation which has been introduced by Hickernell (2002) for $b=2$ and Goda, Suzuki and Yoshiki (2013) for arbitrary $b\in \nat$, $b\ge 2$. We prove that both the minimum Niederreiter-Rosenbloom-Tsfasman weight and the minimum Dick weight of folded Hammersley point sets are large enough to achieve the best possible order of $L_p$ discrepancy for all $1\le p\le \infty$.
\end{abstract}

\section{Introduction and the main result}
\label{Intro}
Let $P\subset [0,1)^2$ be a finite point set in the half-open unit-square. For $\bsy=(y_1,y_2)\in [0,1)^2$, we denote by $[\bszero,\bsy)$ the anchored axis-parallel rectangle $[0,y_1)\times [0,y_2)$. The local discrepancy function $\Delta_P:[0,1)^2\to \RR$ is defined as
  \begin{align*}
    \Delta_P(\bsy) := \frac{1}{|P|}\sum_{\bsx\in P}\chi_{[\bszero,\bsy)}(\bsx)-y_1y_2 ,
  \end{align*}
where $|P|$ denotes the cardinality of $P$ counted with multiplicity, and $\chi_{[\bszero,\bsy)}$ denotes the characteristic function of $[\bszero,\bsy)$. Then, for $1\le p<\infty$, we define the $L_p$ discrepancy of $P$ as the $L_p$-norm of $\Delta_P$, that is,
  \begin{align*}
    L_p(P) := \left( \int_{[0,1)^2}\lvert \Delta_P(\bsy)\rvert^p \rd \bsy\right)^{1/p} .
  \end{align*}
For $p=\infty$, we define the $L_{\infty}$ discrepancy of $P$ as
  \begin{align*}
    L_{\infty}(P) := \sup_{\bsy\in [0,1)^2}\lvert \Delta_P(\bsy)\rvert .
  \end{align*}
Although the extension of this concept to an arbitrary number of dimensions is obvious, we concentrate on the two-dimensional case in this note.

For $1\le p<\infty$, we know that there exists a constant $c_p>0$ such that for any positive integer $N$ and any point set $P$ with $|P|=N$ we have
  \begin{align*}
    L_p(P) \ge c_p \frac{\sqrt{\log N}}{N} .
  \end{align*}
The proofs were given by Roth \cite{R54} for $p=2$, by Schmidt \cite{S77} for $1<p<\infty$, and by Hal\'{a}sz \cite{H81} for $p=1$. It is also known from \cite{C80,R80} that this lower bound is the best possible, apart from the value of the constant. 

For the endpoint case $p=\infty$, we know from \cite{S72} that there exists a constant $c_{\infty}>0$ such that for any positive integer $N$ and any point set $P$ with $|P|=N$ we have
  \begin{align*}
    L_{\infty}(P) \ge c_{\infty} \frac{\log N}{N} ,
  \end{align*}
which is sharp, apart from the value of the constant. Two-dimensional Hammersley point sets, introduced below, are the well-known example of point sets whose $L_{\infty}$ discrepancy is of best possible order.

Explicit construction of point sets which achieve the best possible order of the $L_p$ discrepancy has long been of great interest. We refer to the works of Chen and Skriganov \cite{CS02} and Skriganov \cite{S06} for the first explicit construction of point sets in fixed dimension with optimal $L_2$ and optimal $L_p$ discrepancy for $1<p<\infty$, respectively. Recently, explicit construction of infinite-dimensional sequences with optimal $L_2$ and optimal $L_p$ discrepancy for $1<p<\infty$ has been provided by Dick and Pillichshammer \cite{DP14} and Dick \cite{D14}, respectively.

In the two-dimensional case, which we consider in this note, the $L_p$ discrepancy of two-dimensional Hammersley point sets and their variants has been extensively studied in the literature, see for example \cite{FP09mm,HZ69,LP01,P02,W75}. We also refer to \cite[Section~6]{B11} for a brief survey on this topic. Here two-dimensional Hammersley point sets are defined as follows.

\begin{defn}\label{def:hps}
For $b\ge 2$, $b\in \nat$ and $m\in \nat$, the two-dimensional Hammersley point set in base $b$ consisting of $b^m$ points is defined by
  \begin{align*}
    P_{m} := \left\{ \left( \frac{a_1}{b}+\dots +\frac{a_m}{b^m}, \frac{a_m}{b}+\dots +\frac{a_1}{b^m} \right) \colon a_i\in \{0,1,\dots,b-1\}\right\} .
  \end{align*}
\end{defn}

It was shown in \cite{HZ69,W75} that the $L_2$ discrepancy of Hammersley point sets in base $b$ is {\em not} of best possible order. This negative result also holds true for the $L_p$ discrepancy for all $1\le p<\infty$, as shown in \cite{FP09mm,P02}. In order to overcome this problem, there have been several variants of Hammersley point sets, such as digitally shifted ones and symmetrized ones, such that the $L_p$ discrepancy of those point sets becomes of best possible order, see \cite[Section~6]{B11}.

This note provides a new variant of Hammersley point sets, which we call {\em folded Hammersley point sets}, whose $L_p$ discrepancy is of best possible order for all $1\le p\le \infty$. Folded Hammersley point sets are defined as follows.

\begin{defn}\label{def:fhps}
For $b\ge 2$, $b\in \nat$ and $m,n\in \nat$ such that $n\ge 2m$, the two-dimensional folded Hammersley point set in base $b$ consisting of $b^m$ points is defined by
  \begin{align*}
    P_{m,\tau_n} & := \Big\{ \Big( \frac{a_2\ominus a_1}{b}+\dots +\frac{a_m\ominus a_1}{b^{m-1}}+\frac{\ominus a_1}{b^m}\dots +\frac{\ominus a_1}{b^n}, \\
    & \qquad \quad \frac{a_{m-1}\ominus a_m}{b}+\dots +\frac{a_1\ominus a_m}{b^{m-1}}+\frac{\ominus a_m}{b^m}\dots +\frac{\ominus a_m}{b^n} \Big) \\
    & \qquad \qquad \colon a_i\in \{0,1,\dots,b-1\} \Big\} ,
  \end{align*}
where $\ominus$ denotes the subtraction modulo $b$, and $\ominus a=0\ominus a$.
\end{defn}

Here we note that $P_{m,\tau_n}$ is obtained by applying the {\em $b$-adic baker's transformation of depth $n$} to $P_m$. For $x\in [0,1)$ whose $b$-adic expansion is given by $x=\sum_{i=1}^{\infty}\xi_ib^{-i}$, which is unique in the sense that infinitely many of the $\xi_i$ are different from $b-1$, the $b$-adic baker's transformation of depth $n$ is defined as
  \begin{align*}
    \tau_{n}(x) := \sum_{i=1}^{n}\frac{\xi_{i+1}\ominus \xi_1}{b^i}.
  \end{align*}
It is straightforward to confirm that $P_{m,\tau_n}=\{\tau_n(\bsx)\colon \bsx\in P_m\}$, where $\tau_n$ is applied componentwise.

\begin{rem}
When $n=\infty$, $\tau_n$ corresponds to the original $b$-adic baker's transformation introduced in \cite{GSY1}, which is a generalization of the baker's transformation given by Hickernell \cite{H02} for $b=2$. The term ``folded'' originates from the fact that $\tau_{\infty}$ emulates how baker stretches and folds bread dough for $b=2$, see \cite{H02}. Due to some technical reason, however, we only consider the case that $n$ is finite in this note.
\end{rem}

Now we are ready to introduce the main result of this note.

\begin{thm}\label{thm:main}
For $b\ge 2$, $b\in \nat$ and $m,n\in \nat$ such that $n\ge 2m$, let $P_{m,\tau_n}$ be the folded Hammersley point set in base $b$. Then, for all $1\le p\le \infty$, the $L_p$ discrepancy of $P_{m,\tau_n}$ is of the best possible order. Namely, there exists a constant $C_p>0$ such that
  \begin{align*}
    L_p(P_{m,\tau_n}) \le C_p \frac{\sqrt{m}}{b^m},
  \end{align*}
for $1\le p<\infty$, and
  \begin{align*}
    L_{\infty}(P_{m,\tau_n}) \le C_{\infty} \frac{m}{b^m}.
  \end{align*}
\end{thm} 

\section{Preliminaries}

Throughout this note, we shall use the following notation. Let $\nat$ be the set of positive integers, and let $\nat_0=\nat\cup \{0\}$. The operators $\oplus$ and $\ominus$ denotes the addition and subtraction modulo $b$, respectively. $\ZZ_b$ denotes a finite ring containing $b$ elements. For simplicity, we identify the elements of $\ZZ_b$ with the integers $0,1,\dots,b-1\in \ZZ$.

\subsection{Digital nets}

For $b\ge 2$, $b\in \nat$ and $m,n\in \nat$ with $m\le n$, let $C_1,C_2$ be $n\times m$ matrices over $\ZZ_b$. For $a_1,\dots,a_m\in \ZZ_b$, let $\vec{a}$ denote the vector $(a_1,\dots,a_m)^{\top}$ and let $a=a_1+a_2b+\dots+a_mb^{m-1}$. For $j=1,2$, we define
  \begin{align*}
    \vec{y}_{a,j} = C_j \vec{a} ,
  \end{align*}
where $\vec{y}_{a,j}=(y_{a,j,1},\dots,y_{a,j,n})^{\top}\in \ZZ_b^{n}$, and
  \begin{align*}
    x_{a,j} = \frac{y_{a,j,1}}{b}+\dots+\frac{y_{a,j,n}}{b^n} \in [0,1).
  \end{align*}
Then the point set $P=\{\bsx_0,\dots,\bsx_{b^m-1}\}$, where $\bsx_{a}=(x_{a,1},x_{a,2})$, is called a (two-dimensional) digital net (over $\ZZ_b$ with generating matrices $C_1,C_2$).

From Definition \ref{def:hps}, it is obvious that the two-dimensional Hammersley point set $P_m$ is a digital net over $\ZZ_b$ with generating matrices of size $m\times m$
  \begin{align*}
    C_1 = \left(
    \begin{array}{cccc}
      1 & 0 & \cdots & 0 \\
      0 & 1 & \cdots & 0 \\
      \vdots & \vdots & \ddots & \vdots \\
      0 & 0 & \cdots & 1 
    \end{array} \right) ,
    C_2 = \left(
    \begin{array}{cccc}
      0 & \cdots & 0 & 1 \\
      0 & \cdots & 1 & 0 \\
      \vdots & \text{\reflectbox{$\ddots$}} & \vdots & \vdots \\
      1 & \cdots & 0 & 0 
    \end{array} \right) .
  \end{align*}
From Definition \ref{def:fhps}, to which we apply the identity $a\ominus c= a\oplus (b-1)c \in \ZZ_b$ for all $a,c\in \ZZ_b$, we can see that the two-dimensional folded Hammersley point set $P_{m,\tau_n}$ is a digital net over $\ZZ_b$ with generating matrices of size $n\times m$
  \begin{align}\label{eq:matrix_fhps}
    C_1 = \left(
    \begin{array}{ccccc}
      b-1 & 1 & 0 & \cdots & 0 \\
      b-1 & 0 & 1 & \cdots & 0 \\
      \vdots & \vdots & \vdots & \ddots & \vdots \\
      b-1 & 0 & 0 & \cdots & 1 \\
      b-1 & 0 & 0 & \cdots & 0 \\
      \vdots & \vdots & \vdots & \ddots & \vdots \\
      b-1 & 0 & 0 & \cdots & 0 
    \end{array} \right) ,
    C_2 = \left(
    \begin{array}{ccccc}
      0 & \cdots & 0 & 1 & b-1 \\
      0 & \cdots & 1 & 0 & b-1 \\
      \vdots & \text{\reflectbox{$\ddots$}} & \vdots & \vdots & \vdots \\
      1 & \cdots & 0 & 0 & b-1 \\
      0 & \cdots & 0 & 0 & b-1 \\
      \vdots & \text{\reflectbox{$\ddots$}} & \vdots & \vdots & \vdots \\
      0 & \cdots & 0 & 0 & b-1 
    \end{array} \right) .
  \end{align}

For a digital net $P$, its dual net $P^{\perp}$ is defined as follows.
\begin{defn}\label{def:dual}
For $m,n\in \nat$ with $m\le n$, let $P\subset [0,1)^2$ be a digital net over $\ZZ_b$ with generating matrices $C_1,C_2$ of size $n\times m$. Then its dual net $P^{\perp}\subset \nat_0^2$ is defined as
  \begin{align*}
    P^{\perp} := \{(k_1,k_2)\in \nat_0^2\colon C_1^{\top}\vec{k}_1\oplus C_2^{\top}\vec{k}_2=\bszero \},
  \end{align*}
where we use the following notation: For $k\in \nat_0$ whose $b$-adic expansion is given by $k=\sum_{i=0}^{\infty}\kappa_ib^i$, which is actually a finite expansion, we write $\vec{k}=(\kappa_0,\kappa_1,\dots,\kappa_{n-1})^{\top}\in \ZZ_b^n$. The operator $\oplus$ is applied componentwise.
\end{defn}

From the definition of digital nets, it is clear that the properties of a given digital net depend totally on its generating matrices. In the following, we introduce two weight functions, which serve as a quality measure of generating matrices. One is the Niederreiter-Rosenbloom-Tsfasman (NRT) weight function $\mu_1$ given in \cite{N86,RT97}, and the other is the Dick weight function $\mu_2$ given in \cite{D08}.

\subsection{Niederreiter-Rosenbloom-Tsfasman weight}

\begin{defn}\label{def:nrt_weight}
For $k\in \nat$, we denote its $b$-adic expansion by $k=\kappa_1b^{a_1-1}+\kappa_2b^{a_2-1}+\dots+\kappa_vb^{a_v-1}$ such that $\kappa_1,\dots,\kappa_v\in \{1,\dots,b-1\}$ and $a_1>a_2>\dots > a_v>0$. Then the Niederreiter-Rosenbloom-Tsfasman (NRT) weight function $\mu_1$ is defined as
  \begin{align*}
    \mu_1(k) := \begin{cases}
    a_1     & \text{if}\quad k>0, \\
    0       & \text{if}\quad k=0.
    \end{cases}
  \end{align*}
For vectors $\bsk=(k_1,k_2)\in \nat_0^2$, we define $\mu_1(\bsk):=\mu_1(k_1)+\mu_1(k_2)$.
\end{defn}

For a digital net $P$, we define the minimum NRT weight $\rho_1(P)$ as
  \begin{align*}
    \rho_1(P) := \min_{\bsk\in P^{\perp}\setminus \{(0,0)\}}\mu_1(\bsk) .
  \end{align*}
The following lemma shows how the minimum NRT weight connects with a structure of generating matrices of a digital net $P$, see for example \cite[Theorems~4.52 and 7.8]{DP10}.

\begin{lem}\label{lem:NRT}
For $m,n\in \nat$ with $m\le n$, let $P$ be a digital net over $\ZZ_b$ with generating matrices $C_1,C_2$ of size $n\times m$. Let $\rho$ be a positive integer such that for any choice of $d_1,d_2\in \nat_0$ with $d_1+d_2=\rho$, the first $d_1$ row vectors of $C_1$ and the first $d_2$ row vectors of $C_2$ are linearly independent over $\ZZ_b$. Then we have $\rho_1(P)>\rho$.
\end{lem}

The following proposition adapted from a result of Niederreiter \cite[Theorem~3.6]{N87} shows that the $L_{\infty}$ discrepancy of a (two-dimensional) digital net with large minimum NRT weight is of best possible order.

\begin{prop}\label{prop:NRT}
Let $P$ be a digital net over $\ZZ_b$ consisting of $b^m$ points which satisfies $\rho_1(P)> m-t$ for some integer $0\le t\le m$. Then there exists a constant $C_{t}$ which depends only on $t$ such that we have
  \begin{align*}
    L_{\infty}(P) \le C_t \frac{m}{b^m}.
  \end{align*}
\end{prop}

\subsection{Dick weight}

\begin{defn}\label{def:dick_weight}
For $k\in \nat$, we denote its $b$-adic expansion by $k=\kappa_1b^{a_1-1}+\kappa_2b^{a_2-1}+\dots+\kappa_vb^{a_v-1}$ such that $\kappa_1,\dots,\kappa_v\in \{1,\dots,b-1\}$ and $a_1>a_2>\dots > a_v>0$. Then the Dick weight function $\mu_2$ is defined as
  \begin{align*}
    \mu_2(k) = \begin{cases}
    a_1+a_2 & \text{if}\quad v\ge 2, \\
    a_1     & \text{if}\quad v=1, \\
    0       & \text{if}\quad k=0.
    \end{cases}
  \end{align*}
For vectors $\bsk=(k_1,k_2)\in \nat_0^2$, we define $\mu_2(\bsk):=\mu_2(k_1)+\mu_2(k_2)$.
\end{defn}

For a digital net $P$, we define the minimum Dick weight $\rho_2(P)$ as
  \begin{align*}
    \rho_2(P) := \min_{\bsk\in P^{\perp}\setminus \{(0,0)\}}\mu_2(\bsk) .
  \end{align*}
As with Lemma \ref{lem:NRT}, the following lemma shows how the minimum Dick weight connects with a structure of generating matrices of a digital net $P$, see \cite[Chapter~15]{D14} and \cite[Lemma~4.3]{M14}.

\begin{lem}\label{lem:dick}
For $m,n\in \nat$ with $n\ge 2m$, let $P$ be a digital net over $\ZZ_b$ with generating matrices $C_1,C_2$ of size $n\times m$. For $j=1,2$ and $1\le l\le n$, let $\bsc_{j,l}$ denote the $l$-th row vector of $C_j$. Let $\rho$ be a positive integer such that for all $1\le i_{1,v_1}<\dots<i_{1,1}\le 2m$ and $1\le i_{2,v_2}<\dots<i_{2,1}\le 2m$ with $v_j\in \nat_0$ and
  \begin{align*}
    \sum_{l=1}^{\min(v_1,2)}i_{1,l}+\sum_{l=1}^{\min(v_2,2)}i_{2,l}\le \rho ,
  \end{align*}
the vectors $\bsc_{1,i_{1,v_1}},\dots,\bsc_{1,i_{1,1}},\bsc_{2,i_{2,v_2}},\dots,\bsc_{2,i_{2,1}}$ are linearly independent over $\ZZ_b$. Then we have $\rho_2(P)>\rho$.
\end{lem}

Recently, Dick \cite{D14} proved that digital nets over $\ZZ_2$ with large minimum Dick weight achieve the best possible order of the $L_p$ discrepancy for $1<p<\infty$ and for any number of dimensions. His result was generalized more recently by Markhasin \cite{M14} to digital nets over $\ZZ_b$ for arbitrary $b\ge 2$. We specialize the results of Dick \cite[Corollary~2.2]{D14} and Markhasin \cite[Theorem~1.8]{M14} on the $L_p$ discrepancy of digital nets for the two-dimensional case.

\begin{prop}\label{prop:DM}
Let $P$ be a digital net over $\ZZ_b$ consisting of $b^m$ points which satisfies $\rho_2(P)> 2m-t$ for some integer $0\le t\le 2m$. Then for all $1\le p<\infty$ there exists a constant $C_{p,t}$ which depends only on $t$ and $p$ such that we have
  \begin{align*}
    L_p(P) \le C_{p,t} \frac{\sqrt{m}}{b^m}.
  \end{align*}
\end{prop}

\section{Proof of the main result}

In order to prove Theorem \ref{thm:main}, it suffices from Propositions \ref{prop:NRT} and \ref{prop:DM} to prove that both the minimum NRT weight and the minimum Dick weight of the folded Hammersley point sets are large.

In the following, the generating matrices $C_1$ and $C_2$ always refer to those of \eqref{eq:matrix_fhps}. For $j=1,2$ and $1\le l\le n$, we denote by $\bsc_{j,l}$ the $l$-th row vector of $C_j$. We first show the linear independence properties which $C_1$ and $C_2$ holds.

\begin{lem}\label{lem:linear}
For $m,n\in \nat$ such that $n\ge 2m$, let $C_1$ and $C_2$ be the generating matrices of the two-dimensional folded Hammersley point set in base $b$. The following sets of the vectors are linearly independent over $\ZZ_b$:
\begin{enumerate}
\item $\{\bsc_{1,1},\dots,\bsc_{1,r},\bsc_{2,1},\dots,\bsc_{2,m-1-r}\}$ for $0\le r\le m-1$,
\item $\{\bsc_{1,1},\dots,\bsc_{1,m-2},\bsc_{2,r}\}$ and $\{\bsc_{2,1},\dots,\bsc_{2,m-2},\bsc_{1,r}\}$ for $0\le r\le m-2$,
\item $\{\bsc_{1,1},\dots,\bsc_{1,r},\bsc_{2,1},\dots,\bsc_{2,m-2-r},\bsc_{j,s}\}$ for $j=1,2$, $0\le r\le m-2$, and $m-1\le s\le n$,
\item $\{\bsc_{1,1},\dots,\bsc_{1,r_{1,2}},\bsc_{1,r_{1,1}},\bsc_{2,1},\dots,\bsc_{2,r_{2,2}},\bsc_{2,r_{2,1}}\}$ for $0< r_{1,2}<r_{1,1}\le m-2$ and $0< r_{2,2}<r_{2,1}\le m-2$ such that $r_{1,1}+r_{1,2}+r_{2,1}+r_{2,2}\le 2m-3$.
\end{enumerate}
\end{lem}

\begin{proof}
Since the proofs of the first three items follow essentially the same argument, we only give the proof for Item 1. When $r=0$  ($r=m-1$, resp.), it is trivial that $\bsc_{1,1},\dots,\bsc_{1,m-1}$ ($\bsc_{2,1},\dots,\bsc_{2,m-1}$, resp.) are linearly independent over $\ZZ_b$. For $0< r< m-1$, let us consider the system of equations
  \begin{align}\label{eq:system1}
    a_{1,1}\bsc_{1,1}\oplus \dots\oplus a_{1,r}\bsc_{1,r}\oplus a_{2,1}\bsc_{2,1}\oplus \dots\oplus a_{2,m-1-r}\bsc_{2,m-1-r}=\bszero ,
  \end{align}
where $a_{1,1},\dots,a_{1,r},a_{2,1},\dots,a_{2,m-1-r}\in \ZZ_b$, and $\bszero$ denotes the vector consisting of $m$ zeros. Then we have
  \begin{align*}
    & a_{1,1}\oplus \dots \oplus a_{1,r} = 0 , \\
    & a_{1,1} = \dots =a_{1,r-1} = 0, \\
    & a_{1,r}\oplus a_{2,m-1-r} = 0, \\
    & a_{2,1}=\dots = a_{2,m-2-r} = 0, \\
    & a_{2,1}\oplus \dots \oplus a_{2,m-1-r} = 0 ,
  \end{align*}
from which it is obvious that only one solution $a_{1,1}=\dots=a_{1,r}=a_{2,1}=\dots=a_{2,m-1-r}=0$ satisfies \eqref{eq:system1}. Thus, $\bsc_{1,1},\dots,\bsc_{1,r},\bsc_{2,1},\dots,\bsc_{2,m-1-r}$ are linearly independent over $\ZZ_b$.

Next we give the proof for Item 4. The system of equations is given as
  \begin{align}\label{eq:system2}
    & a_{1,1}\bsc_{1,1}\oplus \dots\oplus a_{1,r_{1,2}}\bsc_{1,r_{1,2}}\oplus a_{1,r_{1,1}}\bsc_{1,r_{1,1}} \nonumber \\
    & \oplus a_{2,1}\bsc_{2,1}\oplus \dots\oplus a_{2,r_{2,2}}\bsc_{2,r_{2,2}}\oplus a_{2,r_{2,1}}\bsc_{2,r_{2,1}}=\bszero ,
  \end{align}
where $a_{1,1},\dots,a_{1,r_{1,2}},a_{1,r_{1,1}},a_{2,1},\dots,a_{2,r_{2,2}},a_{2,r_{2,1}}\in \ZZ_b$. If $r_{1,1}+r_{2,2}\ge m-1$ and $r_{1,2}+r_{2,1}\ge m-1$, the condition $r_{1,1}+r_{1,2}+r_{2,1}+r_{2,2}\le 2m-3$ does not hold. Thus, either $r_{1,1}+r_{2,2}$ or $r_{1,2}+r_{2,1}$ must be less than $m-1$.

Let us consider the case $r_{1,1}+r_{2,2}<m-1$ and $r_{1,2}+r_{2,1}<m-1$ first. If $r_{1,1}+r_{2,1}\le m-1$, it is trivial from Item 1 that the vectors $\bsc_{1,1},\dots,\bsc_{1,r_{1,1}},\bsc_{2,1},\dots,\bsc_{2,r_{2,1}}$ are linearly independent over $\ZZ_b$. Thus, we may suppose $r_{1,1}+r_{2,1}> m-1$. From \eqref{eq:system2}, we have
  \begin{align*}
    & a_{1,1}\oplus \dots \oplus a_{1,r_{1,2}}\oplus a_{1,r_{1,1}} = 0 , \\
    & a_{1,1} = \dots =a_{1,r_{1,2}} = a_{1,r_{1,1}} = 0, \\
    & a_{2,1}=\dots = a_{2,r_{2,2}} = a_{2,r_{2,1}} = 0, \\
    & a_{2,1}\oplus \dots \oplus a_{2,r_{2,2}}\oplus a_{2,r_{2,1}} = 0 .
  \end{align*}
Then, it is obvious that $a_{1,1}=\dots=a_{1,r_{1,2}}=a_{1,r_{1,1}}=a_{2,1}=\dots=a_{2,r_{2,2}}=a_{2,r_{2,1}}=0$, which implies the linear independence of the row vectors.

Finally, let us consider the case $r_{1,1}+r_{2,2}\ge m-1$ and $r_{1,2}+r_{2,1}<m-1$. From \eqref{eq:system2}, we have
  \begin{align*}
    & a_{1,1}\oplus \dots \oplus a_{1,r_{1,2}}\oplus a_{1,r_{1,1}} = 0 , \\
    & a_{1,1} = \dots =a_{1,r_{1,2}} = 0, \\
    & a_{1,r_{1,1}}\oplus a_{2,m-1-r_{1,1}} = 0, \\
    & a_{2,1}=\dots = a_{2,m-1-r_{1,1}-1} = 0, \\
    & a_{2,m-r_{1,1}}=\dots = a_{2,r_{2,2}} = 0, \\
    & a_{2,r_{2,1}} = 0, \\
    & a_{2,1}\oplus \dots \oplus a_{2,r_{2,2}}\oplus a_{2,r_{2,1}} = 0 .
  \end{align*}
Thus, only one solution $a_{1,1}=\dots=a_{1,r_{1,2}}=a_{1,r_{1,1}}=a_{2,1}=\dots=a_{2,r_{2,2}}=a_{2,r_{2,1}}=0$ satisfies \eqref{eq:system2}, so that the linear independence of the row vectors is shown. The case $r_{1,1}+r_{2,2}< m-1$ and $r_{1,2}+r_{2,1}\ge m-1$ can be proven in the same way.
\end{proof}

Using Lemma \ref{lem:NRT} with $\rho=m-1$ and Item 1 of Lemma \ref{lem:linear}, it is straightforward to prove that the minimum NRT weight of folded Hammersley point sets in base $b$ is large. Hence, we omit the proof.

\begin{lem}\label{lem:NRT_fhps}
For $m,n\in \nat$ such that $n\ge 2m$, let $P_{m,\tau_n}$ be the two-dimensional folded Hammersley point set in base $b$ consisting of $b^m$ points. We have $$\rho_1(P_{m,\tau_n})> m-1.$$
\end{lem}

Next we prove that the minimum Dick weight of folded Hammersley point sets in base $b$ is also large, by using the results of Lemmas \ref{lem:dick} and \ref{lem:linear}.

\begin{lem}\label{lem:dick_fhps}
For $m,n\in \nat$ such that $n\ge 2m$, let $P_{m,\tau_n}$ be the two-dimensional folded Hammersley point set in base $b$ consisting of $b^m$ points. We have $$\rho_2(P_{m,\tau_n})> 2m-3.$$
\end{lem}

\begin{proof}
Let $\rho=2m-3$ in Lemma \ref{lem:dick}. Then it suffices to show that for all $1\le i_{1,v_1}<\dots<i_{1,1}\le 2m$ and $1\le i_{2,v_2}<\dots<i_{2,1}\le 2m$ such that
  \begin{align}\label{eq:dick_condition}
    s_1+s_2\le \rho \quad \text{where}\quad s_1:=\sum_{l=1}^{\min(v_1,2)}i_{1,l}\quad \text{and}\quad s_2:=\sum_{l=1}^{\min(v_2,2)}i_{2,l},
  \end{align}
the vectors $\bsc_{1,i_{1,v_1}},\dots,\bsc_{1,i_{1,1}},\bsc_{2,i_{2,v_2}},\dots,\bsc_{2,i_{2,1}}$ are linearly independent over $\ZZ_b$. We consider the following three disjoint cases: the case $s_1=0, s_2>0$, the case $s_1>0, s_2=0$, and the case $s_1,s_2>0$. Since the proof for the second case follows exactly the same arguments used for the first case, we only give the proofs for the first and third cases.

Let us consider the case $s_1=0, s_2>0$ first. Since the case $v_2=1$ is trivial, we suppose $v_2\ge 2$. If $i_{2,1}\le m-1$, the vectors $\bsc_{2,i_{2,v_2}},\dots,\bsc_{2,i_{2,1}}$ are linearly independent over $\ZZ_b$ for any $1\le i_{2,v_2}<\dots<i_{2,1}$, since $\bsc_{2,1},\dots,\bsc_{2,m-1}$ are linearly independent over $\ZZ_b$ as in Item 1 of Lemma \ref{lem:linear}. Thus, we may suppose $i_{2,1}> m-1$ below. Since $i_{2,1}+i_{2,2}\le \rho$, we must have $i_{2,2}< m-2$. Then, from Item 3 of Lemma \ref{lem:linear}, the vectors $\bsc_{2,i_{2,v_2}},\dots,\bsc_{2,i_{2,2}},\bsc_{2,i_{2,1}}$ are linearly independent over $\ZZ_b$ for any $1\le i_{2,v_2}<\dots<i_{2,2}< m-2$. Thus, the proof for the first case is complete.

Let us move on to the case $s_1,s_2>0$. If $i_{1,1},i_{2,1}\ge m-1$, then we have $s_1+s_2>\rho$, which violates the condition \eqref{eq:dick_condition}. Thus, either $i_{1,1}$ or $i_{2,1}$ must be less than $m-1$. In the following, we further split the case $s_1,s_2>0$ into three disjoint cases: the case $i_{1,1}< m-1\le i_{2,1}$, the case $i_{2,1}< m-1\le i_{1,1}$, and the case $i_{1,1},i_{2,1}< m-1$.

Let $i_{1,1}< m-1\le i_{2,1}$ first. Then from \eqref{eq:dick_condition} we have $0<i_{1,1}\le m-2$ if $v_2=1$ and $0<i_{1,1}+i_{2,2}\le m-2$ if $v_2\ge 2$. For $v_2=1$, $\bsc_{1,1},\dots,\bsc_{1,m-2},\bsc_{2,i_{2,1}}$ are linearly independent over $\ZZ_b$ as in Item 3 of Lemma \ref{lem:linear}. Thus, it means that the vectors $\bsc_{1,i_{1,1}},\dots,\bsc_{1,i_{1,v_1}},\bsc_{2,i_{2,1}}$ are linearly independent over $\ZZ_b$. For $v_2\ge 2$, it suffices to prove that $\bsc_{1,1},\dots,\bsc_{1,r},\bsc_{2,1},\dots,\bsc_{2,m-2-r},\bsc_{2,i_{2,1}}$ are linearly independent over $\ZZ_b$ for any $1\le r\le m-3$, which has been already proven in Item 3 of Lemma \ref{lem:linear}. Thus the proof for the case $i_{1,1}< m-1\le i_{2,1}$ is complete. The same argument can be applied to the case $i_{2,1}< m-1\le i_{1,1}$.

Let $i_{1,1},i_{2,1}< m-1$ next. Suppose that $v_1=1$. For given $1\le i_{1,1}<m-1$, $\bsc_{2,1},\dots,\bsc_{2,m-2}$ and $\bsc_{1,i_{1,1}}$ are linearly independent over $\ZZ_b$ as in Item 2 of Lemma \ref{lem:linear}. The same argument can be applied to the case $v_2=1$. Finally, suppose $v_1,v_2\ge 2$. Here we require that $i_{1,1}+i_{1,2}+i_{2,1}+i_{2,2}\le \rho \;(=2m-3)$ and $i_{1,1},i_{2,1}< m-1$. From this requirement for $i_{1,1},i_{1,2},i_{2,1},i_{2,2}$ and the result of Item 4 in Lemma \ref{lem:linear}, $\bsc_{1,1},\dots,\bsc_{1,i_{1,2}},\bsc_{1,i_{1,1}},\bsc_{2,1},\dots,\bsc_{2,i_{2,2}},\bsc_{2,i_{2,1}}$ are linearly independent over $\ZZ_b$. Thus, the proof for the case $i_{1,1},i_{2,1}< m-1$ is complete.
\end{proof}

\begin{rem}
Our proof used in this note is based on the fact that folded Hammersley point sets are digital nets with good generating matrices. As far as the author knows, other variants of Hammersley point sets whose $L_p$ discrepancy is of best possible order, such as digitally shifted ones and symmetrized ones, are no longer digital nets,  so that our proof technique cannot be applied to them.
\end{rem}


\end{document}